\documentclass[11pt,a4]{article}
\usepackage{authblk}
\usepackage{amsmath,amsfonts,amsthm,amssymb,graphicx}
\usepackage{color}
\usepackage{hyperref}
\newtheorem{theorem}{Theorem}%[section]
\newtheorem{lemma}[theorem]{Lemma}%[section]
\newtheorem{remark}[theorem]{Remark}%[section]

\newcommand*{\B}[1]{\mathbf{#1}}
\newcommand*{\CR}{\mbox{\scriptsize CR}}
\newcommand*{\bu}{\mathbf{u}}
\newcommand*{\bv}{\mathbf{v}}
\newcommand*{\bw}{\mathbf{w}}
\newcommand*{\bp}{\mathbf{p}}

\newcommand*{\bff}{\mathbf{f}}
\newcommand*{\bX}{\mathbf{X}}

\newcommand*{\bV}{\mathbf{V}}
\newcommand*{\bU}{\mathbf{U}}
\newcommand*{\bRT}{\mathbf{RT}}
\newcommand*{\Div}{\mbox{div}\:}
\newcommand*{\Divv}{\mbox{div}~}

  \title{Computer-assisted proof for the stationary solution existence of the Navier--Stokes equation over 3D domains
  \thanks{The first author is supported by Japan Society for the Promotion of Science, Grant-in-Aid for Scientific Research (B) 20H01820, 21H00998, and Grant-in-Aid for Scientific Research (C) 18K03411.
The second author is supported by Grant-in-Aid for Scientific Research (C) 18K03434, 21K03378. The last author is supported by JST CREST Grant Number JPMJCR14D4, Japan.}
}
  
  \author[1]{Xuefeng LIU}
  \affil[1]{Niigata University\thanks{xfliu@math.sc.niigata-u.ac.jp}}

  \author[2]{Mitsuhiro T. NAKAO}
  \affil[2,3]{Waseda University}

  \author[3]{Shin'ichi OISHI}
\begin{document}

 \maketitle
  %~ \thanks{This work was supported by KAKENHI (18K03411,16Ha03950).}
 %%%% Abstract

\begin{abstract}
This paper proposes a computer-assisted solution existence verification method for the stationary Navier--Stokes equation over general 3D domains. The proposed method verifies that the exact solution as the fixed point of the Newton iteration exists around the approximate solution through rigorous computation and error estimation.
The explicit values of quantities required by applying the fixed point theorem are obtained by utilizing newly developed quantitative error estimation for finite element solutions to boundary value problems and eigenvalue problems of the Stokes equation.

\end{abstract}

%   \subjclass[2000]{MSC-35Q30, MSC-65M60}
%   \keywords{Navier--Stokes equation, verified computing, finite element method, computer-assisted proof}

  %%%% Body
\section{Introduction}
As a new approach to investigating the solution existence of nonlinear equation systems, verified computing has attracted the attention of researchers in the fields of pure mathematics and scientific computing. In the past decades, there have been several fundamental results with regard to milestones of solution verification for nonlinear equations (refer to the early work of  M. Nakao, M. Plum, and S. Oishi \cite{nakao1988numerical, Plum1991,Oishi1995} and the newly published book \cite{nakao2019numerical_book}).

The fixed-point theorem is a fundamental tool in solution verification, and several variations of the theorem being used to solve practical problems exist. For example, the Newton--Kantorovich theorem was applied to the solution verification for a semi-linear elliptic partial differential equation in \cite{Takayasu2013}, and it is also utilized in this study. A similar approach was proposed by Plum \cite{Plum2009}, which relaxes the condition on the continuity of the Fr\'echet derivative of a functional. An early approach was provided by Nakao \cite{nakao1988numerical,NAKAO1992} to subdivide the function space into a finite-dimensional part and an orthogonal part to form the fixed-point formulation.

\medskip

In this paper, we examine the solution verification for the Navier--Stokes equation. 
Owing to the existence of a nonlinear convection term in the equation,  it is challenging work to study the solution of the Navier--Stokes equation.
In 1999, Watanabe--Yamamoto--Nakao \cite{Watanabe_etal1999} considered the Navier--Stokes equation over a two-dimensional (2D) square domain and provided a successful solution verification case. The approach utilized the {\em a priori} error proposed by Nakao, Yamamoto, and Watanabe \cite{Nakao_etal1998} to solve residue error of the non-diverging part of the finite element method (FEM) approximation to the exact solution. 
The {\em a priori} error estimation of \cite{Nakao_etal1998} is based on the constant appearing in the Korn inequality. Since the explicit value of the constant is not available for 3D domains, it is difficult to follow the approach of \cite{Nakao_etal1998} to solve problems on 3D domains.

Other researchers have examined solution verification opportunities for the Navier--Stokes equation.
In \cite{kobayashi2013computer}, K. Kobayashi reported a proof for the global uniqueness of Stokes' wave of the extreme form. 
In 2020, at the seminar of the Centre de Recherches Mathematiques Computer-assisted Mathematical Proofs (CRM CAMP) in Nonlinear Analysis, J. Wunderlich reported a computer-assisted existence proof for Navier--Stokes equations on an unbounded strip with an obstacle. Such an approach is based on the homotopy eigenvalue estimation method developed by M. Plum that works well for unbounded domains (refer to \cite{pacella2017computer}).
In \cite{van2021spontaneous}, J.B. van den Berg, M. Breden, JP. Lessard et al. reported a constructive proof of the existence of periodic 
orbits in the forced autonomous Navier--Stokes equations on a three-torus, and the solution verification succeeded for a reduced 2-dimensional space under the symmetry condition of the solution.

\medskip

In this paper, we follow the frame of Newton--Kantorovich's theorem and utilize the newly developed {\em a priori} error estimation for strictly divergence-free FEM approximation \cite{Liu-2021} to consider the solution verification for the stationary Navier--Stokes equation over general 3D domains.
Regarding the kernel problems of applying Newton--Kantorovich's theorem,  the following schemes are used:
\begin{itemize}
 \item [1)]    
 We apply the Scott-Vogelius finite element space to obtain a divergence-free approximate solution to the Navier--Stokes equation.
 To provide the {\em a priori} error estimation of the projection that maps the solution-existing space to the Scott-Vogelius space, 
 the hypercircle-based {\em a priori} error estimation method developed by Liu \cite{Liu-2021} is adopted. 
 This method is a generalization of one proposed by Liu--Oishi for Poisson's equation \cite{Liu+Oishi2012}, which inherits the idea of Kikuchi \cite{Kikuchi+Saito2007} for the {\em a posterior } error estimation.
\item [2)] To provide rigorous eigenvalue estimation for differential operators over a 3D domain, the guaranteed eigenvalue evaluation method proposed by Liu \cite{Liu2015} is used, which can deal with domains of general shapes in a concise and uniform approach.
 \item [3)]  To bound the norm of the inverse of a differential operator, the algorithm based on the fixed-point theorem \cite{Watanabe_etal1999} is utilized.  Moreover, in this paper, we propose a new algorithm to process the divergence-free condition and obtain a direct evaluation for the quantity $\tau$, which is required by the approach of \cite{Watanabe_etal1999}. 
\end{itemize}
  
The rest of the paper is organized as follows. In \S 2, we introduce the notation for function spaces and the problem setting and introduce the Newton--Kantorovich theorem along with the kernel sub-problems to overcome.
In \S 3, the approach to sub-problems is described in a detailed way. In \S 4, the summarization of the implementation of the Newton--Kantorovich theorem in the solution verification is provided. In \S 5, a successful case of solution verification in a 3D domain with a hole inside is reported.

\section{Function spaces and the methodology for solution verification}
\label{sec:ns-eq}
%
%\subsection{Equation in operator formulation} 

In this paper, we concern the stationary Navier--Stokes equation over a 3D domain $\Omega$ of general shape:
\begin{equation}
    \label{eq:ns-eq}
-\epsilon \Delta \mathbf{u} +  (\mathbf{u}\cdot \nabla )\mathbf{u} + \nabla p= \mathbf{f}, ~ \mbox{div }\mathbf{u}=0 \mbox{ in } \Omega, \mathbf{u}=0\mbox{ on }\partial \Omega \:.
\end{equation}
Here, $\epsilon$ is the viscosity coefficient of the fluid, 
 \(\mathbf{f}:\Omega \rightarrow \mathbb{R}^3\) is an applied body
force, \(\mathbf{u} : \Omega \rightarrow \mathbb{R}^3 \) is the velocity vector and \(p:\Omega \rightarrow \mathbb{R} \) is the pressure. In addition, symbols \(\Delta\) and \(\nabla\) denote the
Laplacian and the gradient operators, respectively.
For a vector field $\bu=(u_1,u_2,u_3):\Omega \to \mathbb{R}^3$, its divergence is denoted by $\Divv \bu:=\partial_x u_1 +\partial_y  u_2 +\partial_z  u_3$.

\medskip

Let us describe the functions spaces to be used to study equation (\ref{eq:ns-eq}). 
Define function space $\bV$ by 
\begin{equation}
\label{def:V}
\bV=\{ \bv \in \left( H_0^1(\Omega)\right)^3| \mbox{ div } \bv = 0\}\:,
\end{equation}
along with inner product and norm
$$
(\bu,\bv):=\int_\Omega \nabla \bu \cdot \nabla \bv ~ \mbox{d}\Omega, ~~ \|\bu\|_{\bV} :=\sqrt{(\bu,\bu)} \:.
$$
The dual space of $\bV$ is denoted by $\bV^\ast$. 
The definition of Stokes quation also involves the space $L^2(\Omega)$ and $\bX= (L^2(\Omega))^3$, the norm of which are denoted by $\|\cdot\|_{L^2(\Omega)}$ or just $\|\cdot\|$.

\medskip

%Since the compact embedding holds for $H_0^1(\Omega)\hookrightarrow L^2(\Omega)$, we have the embedding such that $\bV \hookrightarrow  \bX$.
%A function $\bff\in \bX$ can be regarded as a functional over $\bV$, i.e., 
%$\langle \bff , \cdot \rangle =(\bff, \cdot)$.

Define $\mathcal{A} : \bV\to  \bV^\ast$  by,
$$
\langle \mathcal{A} [\bu],\bv \rangle := ( \epsilon \nabla \bu, \nabla \bv) \mbox{ for } \bu, \bv \in \bV \:.
$$
% Since 
% , the Ritz representation theorem assures the mapping 
% $\mathcal{A}^{-1}$ such that 
% $$
% \langle \mathcal{A}^{-1}\bff, \bv \rangle = (\bff,v) \mbox{ for all } v \in \bV \:.
% $$
%
With $\bff \in \bX$, defined $\mathcal{N} : \bV\to \bV^\ast$ by
$$
\langle \mathcal{N}[\bu],\bv \rangle :=  (\bff , \bv) - ( (\bu\cdot \nabla )\bu ,\bv)   \mbox{ for } \bu, \bv \in \bV \:.
$$
Let $\mathcal{F}:=\mathcal{A} - \mathcal{N}$.
% Define $\mathcal{F} : \bV\to  \bV^\ast$ by
% $\langle \mathcal{F}[\bu],\bv \rangle := \langle \mathcal{A}[\bu],\bv \rangle  +  \langle \mathcal{N}[\bu],\bv \rangle $. 
That is,
$$
\langle \mathcal{F}[\bu],\bv \rangle 
= ( \epsilon \nabla \bu, \nabla \bv) + ( (\bu\cdot \nabla )\bu ,\bv) - (\bff ,\bv)\:.
$$
Then the Navier--Stokes equation can be formulated as the equation of functional.
\begin{equation}
\label{eq:variational-eq-ns}
\mathcal{F}[\bu]=\mathcal{A}[\bu] - \mathcal{N}[\bu] =  0\:.
\end{equation}
The Fr\'echet derivative of $\mathcal{F}:\bV\to \bV^\ast$ 
at $\hat{\bu} \in \bV$ can be given by
$$
\langle \mathcal{F}'[\hat{\bu}] \bu, \bv \rangle = \langle \mathcal{A} [\bu],\bv \rangle - \langle \mathcal{N}'[\hat{\bu}]\bu, \bv \rangle 
=\epsilon (\nabla \bu, \nabla \bv) + ((\hat{\bu} \cdot \nabla ) \bu,\bv ) + ((\bu \cdot \nabla ) \hat{\bu}, \bv )\:.
$$ 

The following is the theoretical result from Girault-Raviart's book for exploring the solution existence of the Navier--Stokes equation.
\begin{theorem}{Theorem 2.2 (Chapter IV) of \cite{girault2012finite} }
Let ${N}$ and $\|f\|_{V^\ast}$ be defined by
$$
{N} := \sup_{u,v,w\in V} \frac{\int_\Omega ( w\cdot \nabla )u v \mbox{\:\em d} \Omega}{\|u\|_V \cdot \|v\|_V  \cdot \|w\|_V },
\quad 
\|f\|_{V^\ast}=\sup_{v\in V} \frac{(f,v)}{\|v\|_V } \:.
$$
If ${N} \cdot \|f\|_{V^\ast}/\epsilon^2 <1 $, then the Navier--Stokes equation has a unique solution in $V$.
\end{theorem}

This theoretical result helps to show solution existence if $\epsilon$ is not very small. 
In the upcoming numerical example section, we show an example with a small $\epsilon$ for which this theory fails to draw a conclusion about the solution existence; however, our proposed method works well.

\subsection{The main theorems for computer-assisted solution verification}

Below is the fundamental theorem used by our algorithm to verify the solution's existence. We cite the theorem in a general setting of spaces. 
Let $V$ be a Hilbert space and $\mathcal{F}:V\to V^\ast$ be a functional.
The following theorem provides a basic frame for investigating the solution of the equation $F[u]=0$.

\begin{theorem}[Newton--Kantorovich's theorem \cite{Newton--Kantorovich}]
\label{thm:Newton--Kantorovich}
Given $\hat{u} \in V$, assume that ${\mathcal F}'[\hat{u}]$ is regular and the following inequality holds with constant $\alpha>0$:
$$
\|{\mathcal F}'[\hat{u}]^{-1}{\mathcal F}[\hat{u}]\|_V \le \alpha\:.
$$
Let $B(\hat{u},2\alpha)(\subset V)$ be the closed ball centered at $\hat{u}$ and the radius being $2\alpha$. Assume that the following inequality holds for an open ball $D$ satisfying $B(\hat{u},2\alpha) \subset D$ along with the constant $\omega$,
$$
\| \mathcal{F}'[\hat{u}]^{-1} (\mathcal{F}'[v]-\mathcal{F}'[w]) \|_{V,V} 
\le \omega \|v-w\|_{V},\quad \forall v,w \in D\:.
$$
If $\alpha \omega \le1/2$ holds, then $\mathcal{F}[u]=0$ has a unique solution in $u\in B(\hat{u},\rho)$, where $\rho$ is given by
$$
\rho := \frac{ 1- \sqrt{1-2\alpha \omega} }{\omega}\:.
$$

\end{theorem}

In the application of the Newton--Kantorovich theorem to solution verification, the following quantities should be estimated explicitly.

\begin{enumerate}
\item [1)] Norm estimation for the inverse of $\mathcal{F}'[\hat{u}]$ : ~ $ \| \mathcal{F}'[\hat{u}]^{-1} \|_{V^\ast,V} \le K   $.
\item [2)] Residue error of $\hat{u}$ :  ~$\|\mathcal{F}[\hat{u}]\|_{V^\ast} \le \delta $.
\item [3)] Local continuity of $\mathcal{F}'$: ~$  \|\mathcal{F}'[v] - \mathcal{F}'[w]\|_{V,V^\ast} \le G\|v-w\|_{V}, \quad \forall v,w \in D$.
\end{enumerate}

Once the quantities $K$, $\delta$, and $G$ are evaluated, the constant $\alpha$ and $\omega$ can be given as 
$$
\alpha := K\delta,\quad \omega := KG
$$
If $K^2 G \delta \le 1/2$ holds, there exists a unique solution of $\mathcal{F}[u]=0$ in $B(\hat{u},\rho)$.

\medskip

In applying Newton--Kantorovich's theorem to the solution verification for the Navier--Stokes equation, the explicit estimation of $K$ is the most challenging part.  
The evaluation of $K$ is to estimate the norm of the inverse of a differential operator, which reduces to solving eigenvalue problems of operators. 
Since the involved eigenvalue problem is related to a non-self-adjoint differential operator (denoted by $\mathcal{K}$), it is not easy to deal with the eigenvalue problem directly. A choice is to 
apply the idea of M. Plum to consider $\mathcal{K}\cdot\mathcal{K}^\ast$ ($\mathcal{K}^\ast$: the conjugate operator of $\mathcal{K}$); see \S 9 of \cite{nakao2019numerical}. 
Here we turn to the method proposed by M. Nakao to avoid the eigenvalue estimation \cite{Nakao_etal2005}. Recent discussion on improvement of Nakao's method can be found in 
\cite{Sekine2020,Watanabe2020}.
Another approach for estimating $K$ can also be found in  \cite{Oishi1995} of S. Oishi. 
\medskip

Let us introduce Nakao's method in a general problem setting with a linear operator $\mathcal{A}:V\to V^\ast$ and a non-linear operator $\mathcal{N}:V\to V^\ast$ defined over the Hilbert space $V$. Let $\mathcal{F}:=\mathcal{A} - \mathcal{N}$.
Let $X$ be a Hilbert space such that the compact embedding $V \hookrightarrow X$ holds. 
Each $f\in X$ can be regarded as functional over $V$, i.e., $\langle f, v\rangle = (f,v)_X, \forall v \in V$. 
The operator $\mathcal{A}^{-1}:X\to V$ maps  ${f} \in X$ to the solution $u$ of the following variational equation: 
$$
\langle \mathcal{A} u, v \rangle = (f,v)_X, ~ \forall v \in V\:.
$$
Let $V_h$ be a subspace of $V$ and $P_h:V\to V_h$ be the projection under the inner product of $V$. The detailed  selection of $V_h$ in practical soluiton verification is described in the beginning of \S\ref{sec:approach-to-sub-problems}. Suppose the following {\em a priori} error estimation holds for the projection $P_h$: 

\begin{equation}
\label{eq:a_priori_est_epsilon_Ph}
\|(I-P_h) (\mathcal{A}^{-1} {f} ) \|_V \le C_{h,\mathcal{A}} \| {f} \|_X, \quad \forall f \in X \:.
\end{equation}

To give  estimation of  $\|\mathcal{F}'[\hat{u}]^{-1}\|_{V^\ast,V}$, 
one need to consider the mapping between $\phi$ and $u$: $\phi =  \mathcal{F}'[\hat{u}] u$.
Nakao's method proposed in \cite{Nakao_etal2005}  estimates $\|\mathcal{F}'[\hat{u}]^{-1}\|_{V^\ast,V}$ by decomposing $u$ by $P_h u + (I-P_h) u$ and estimating the variation of each part under the mapping $\mathcal{F}'[\hat{u}]$. Below, we quote the result of \cite{Nakao_etal2005} with the notation of this paper.

\begin{theorem}[Estimation of $K$ \cite{Nakao_etal2005}; see also a compact proof in \cite{Takayasu2013}]
\label{thm:nakao}
Suppose the following inequalities hold with quantities $\nu_1,\nu_2,\nu_3$,
\begin{eqnarray}
& & \| {P}_{h} {{\mathcal A}}^{-1} {\mathcal N}'[\hat{u}] \mathbf{u}_c \|_{V} \le \nu_1 \| \mathbf{u}_c \|_{V}, ~ \forall \mathbf{u}_c \in V_h^{\perp}, \\
& & \|\mathcal  N'[\hat{u}] \mathbf{u} \|_{V^\ast} \le \nu_2 \|\mathbf{u}\|_{V}, ~ \forall \mathbf{u} \in V,  \\
& & \| \mathcal N'[\hat{u}] \mathbf{u}_c \|_{V^\ast} \le \nu_3 \| \mathbf{u}_c\|_{\mathbf{V}}, ~ \forall \mathbf{u}_c \in V_h^{\perp}.
\end{eqnarray}
Here, $V_h^{\perp}$ is the orthogonal complement space of $V_h$ in  ${V}$.
Assume that the operator ${P}_{h}(I - {{\mathcal A}}^{-1} {\mathcal N}'[\hat{u}])|_{V_h} : V_h \rightarrow V_h$ is invertible and the following estimation holds along with the constant $\tau$
$$
\left\| \left( {P}_{h}(I - {{\mathcal A}}^{-1}{\mathcal N}'[\hat{u}])|_{V_h}  \right)^{-1} \right\|_{L(V,V)} \le \tau \:.
$$
Define $\kappa := (\tau \nu_1 \nu_2 + \nu_3)C_{h,\mathcal{A}}$. If $\kappa < 1$, then we have 
$$
\|  {\mathcal F}'[\hat{u}]  ^{-1} \|_{V^\ast, V} \le \|R\|_2,
$$
where $\| \cdot \|_2$ is the spectral norm for matrix and 
\begin{eqnarray*}
R:=\frac{1}{1-\kappa}\left( \begin{array}{c c}
\displaystyle
\tau \left( 1-\kappa+ { \tau \nu_1\nu_2} \: C_{h,\mathcal{A}}  \right) &
\displaystyle {\tau \nu_1}\\
\displaystyle    \tau \nu_2 \: C_{h,\mathcal{A}}  & \displaystyle 1
\end{array}
\right)\:.
\end{eqnarray*}
\end{theorem}

\begin{remark}
{\bf Sub-problems in solution verification}
Newton--Kantorovich's theorem along with Nakao's method in Theorem \ref{thm:nakao} have been successfully applied to non-linear PDE problems. However, several sub-problems remained in verifying the solution to the Navier--Stokes equation, especially for the processing of the divergence-free condition.
\begin{itemize}
\item [a)] The {\em a priori} error estimation for the approximate solution to the Stokes equation, especially for 2D or 3D domains of general shapes.
\item [b)] The rigorous eigenvalue estimation for various differential operators. For example, to give the Poincare constant over a divergence-free space, one needs to solve the Stokes operator to have rigorous bounds for the eigenvalues.
\item [c)] The estimation of $\tau$ in Theorem \ref{thm:nakao} requires solving the norm of a linear operator restricted on a subspace of $V$ upon the divergence-free condition. However, the operator does not have an explicit representation matrix because the subspace introduced by the Scott-Vogelius FEM is implicitly defined under the divergence-free constraint conditions.  
\end{itemize}
\end{remark}

\vskip 0.5cm

\section{Approach in solving the three sub-problems}
\label{sec:approach-to-sub-problems}
In this section, we introduce the approaches to solve the three sub-problems mentioned in the previous section.  

Let $\mathcal{T}^h$ be a regular subdivision of $\Omega$ with tetrahedron elements. Let us introduce the finite element spaces that will be utilized in the following discussion.

% As to be mentioned in \S \ref{subsec:mesh-generation}, in constructing a strictly divergence-free subspace by the Scott-Vogelius type FEM space, a further subdiviion of elements  of $\mathcal{T}^h$ will be performed to have refined mesh. 

\paragraph{\em Discontinuous space $X_h$ of degree $d$}
$X_h$ is the set of piecewise polynomial of degree $d$ without the requirement of continuity.  The subspace of $X_h$ with function of zero average is denoted as $X_{h,0}$. Further, we define vector function spaces $\mathbf{X}_h:=\left(X_h\right)^3$ and 
$\mathbf{X}_{h,0}:=\left(X_{h,0}\right)^3$.

\paragraph{\em Conforming FEM space $U_h(\subset \left(H^1(\Omega)\right)^3)$ and $V_h (\subset V)$ of degree $k$.}

\begin{itemize}
\item Let $U_h$ be the set of piecewise polynomials of degree up to $k$, which also belongs to $H^1(\Omega)$. Define $\mathbf{U}_h:=\left({U}_h\right)^3$.
\item Let $U_{h,0}:=\{ u_h \in U_h | u_h = 0 \mbox{ on } \partial \Omega \}$, $\mathbf{U}_{h,0} :=\left({U}_{h,0}\right)^3 $.
\item Let $\mathbf{V}_{h}$ be the subspace  of $\mathbf{U}_{h,0}$ with member function satisfying the divergence-free condition, i.e., 
$\mathbf{V}_{h} = \{\bu_h\in \mathbf{U}_{h,0} ~|~ \mbox{div }\bu_h=0 \}=\mathbf{U}_{h} \cap {V}$.
\end{itemize}

\paragraph{\em Construction of $V_{h}$ }

Generally, it is difficult to construct $\mathbf{V}_{h}$ directly due to the divergence-free condition. We turn to utilize the 
Scott-Vogelius type FEM space,
$$
\mathbf{V}_{h} = \{ \mathbf{v}_h \in \mathbf{U}_{h,0} \:| \: (\mbox{div\:} \mathbf{v}_h, q_h)=0 \:\: \forall q_h \in X_h \}\:,
$$
where the degree $k$ of $V_h$ and the degree $d$ of $X_h$ satisfy $d=k-1$.

\paragraph{\em The Raviart--Thomas FEM space $RT_h$ of degree $m$}

The Raviart--Thomas space of degree $m$ is defined as follows. 
$$
RT_h:=\{ p_h \in H(\mbox{div};\Omega) \:\:|\:\: p_h|_K = (a+dx, b+dy,c+dz), a,b,c,d\in P^m(K) \}\:,
$$
where $P^m(K)$ denotes the set of polynomials with degree up to $m$ on element $K$.
Also, the tensor space with $\bp_h \in RT_h^3(\subset H(\mbox{div};\Omega)^3 )$ is denoted by $\mathbf{RT}_h$.

\medskip

The Crouzeix--Raviart FEM space will also be needed in solving the eigenvalue problems; see the introduction in \S\ref{sec:sub-problem-b}.  
In the discussion below, the selection of $k,d$ and $m$ satisfies $d = m = k$.

\medskip

\paragraph{Porjection error estimation}

Let $\pi_h :L^2(\Omega) \to X_h$ be the $L^2$-projection such that for every $u\in L^2(\Omega)$ over an element $K$ of the mesh, $(\pi_h u)|_K$ 
takes the average of $u$ over $K$.
The following  error estimation of $\pi_h$ holds
\begin{equation}
\label{eq:c_0_h}
\|u -\pi_h u\| \le \widehat{C}_{0,h} \|\nabla u \|, \quad \forall u \in H^1(\Omega)\:.
\end{equation}
Here $\widehat{C}_{0,h}=O(h)$ is the error constant that can be estimated with a concrete value. 
The projection $\pi_h$ can be naturally extended to $(L^2(\Omega))^3$, for which the same notation is used. The following error estmation of $\pi_h$ will be needed in the {\em a priori} error estimation for the Stokes equation. 
\begin{equation}
\label{eq:c_0_h_V}
\|\bu -\pi_h \bu\| \le {C}_{0,h} \|\nabla \bu \|, \quad \forall \bu \in \bV\:.
\end{equation}
It is easy to see that $\widehat{C}_{0,h}$ in \eqref{eq:c_0_h} provides an upper bound of ${C}_{0,h}$, due to the boundary condition and the divergence-free condition applied to $\bu \in \bV$. The concrete upper bounds of the two constants are provided in \S \ref{subsec:mesh-generation}.

\subsection{Sub-problem a): The {\em a priori} error estimation for the FEM solution to the Stokes equation}

The hypercircle method, also named by Prager--Synge's Theorem \cite{prager1947approximations}, 
has been successfully applied to the {\em a priori} error estimation for the Poisson equation \cite{Liu+Oishi2012, Kikuchi+Saito2007}.
Here, in a concise way, we introduce an extended version of the hypercircle and construct the {\em a priori} error estimation for the Stokes equation; for a detailed discussion of this topic, refer to  \cite{Liu-2021}.
\medskip
 
 Let us consdier the Stokes equation over domain $\Omega$: Given $\bff \in (L^2(\Omega))^3$, find \(\mathbf{u} \in \mathbf{V}\) such that,
\begin{equation}
    \label{eq:stokes}
    (\nabla \bu, \nabla \bv) = (\bff, \bv), \quad \forall \bv \in \bV.
\end{equation}
The above equation implies an operator $\Delta_s ^{-1}:(L^2(\Omega))^3\to \bV$, which maps the function $\bff$ to the solution $\bu$ of the Stokes equation.

The Stokes equation \eqref{eq:stokes} can be solved approximately in FEM space $\bV_h$. %By using the projection operator defined in \ref{eq:def-Ph}, 
Let $P_h$ be the projection $P_h:\bV \to \bV_h$ such that, for any $\bv \in \bV$
\begin{equation}
    \label{eq:def-Ph}
    (\nabla (\bv - P_h \bv) , \nabla \bv_h) = 0, \quad \forall \bv_h \in 
\bV_h.
\end{equation}

Next, we show the extended hypercircle for the Stokes equation that helps to construct the {\em a priori} error estimation to $(\bu -\bu_h)$.
\begin{lemma}[Extended Prager-Synge's theorem]
  \label{lemma:prager-synge}
	Given $\bff \in (L^2(\Omega))^3$, let $\bu$ be the solution to (\ref{eq:stokes}) corresponding to $\bff$.
	Suppose that \(\bp \in H(\Div;\Omega)^{3}\) satisfies,
\begin{equation}
  \mbox{\em div } \mathbf{p} + \nabla \phi + \mathbf{f} =0, \text{ for certain } \phi \in H^1(\Omega)\:.  
\end{equation}
%
%Here $\Divv \B{p}=(\Divv p_{1}, \Divv p_{2}, \Divv p_{3})$.
Then, for any \(\mathbf{v} \in \mathbf{V}\), the following Pythagorean equation holds:

\begin{equation}
\label{eq:hypercircle-eq}
\| \nabla \mathbf{u} - \nabla \mathbf{v} \|^2 + \|\nabla \mathbf{u} - \mathbf{p} \|^2 = \|\mathbf{p} - \nabla \mathbf{v} \|^2 \:. 
\end{equation}
\end{lemma}
\begin{proof}
The hypercircle comes from the vanishing of the cross term: 
$$
(\nabla \bu - \bp, \nabla (\bu -\bv ))
= (\bff + \Divv \bp , (\bu -\bv )) 
= (-\nabla \phi, (\bu -\bv ) ) = (\phi, \Divv (\bu - \bv))=0\:,
$$
where the divergence-free condition and the boundary conditions of $\bu$ and $\bv$ are utilized.
\end{proof}    
Let us follow the ideas of \cite{liu-2013-2, Liu-2021} to define the quantity $\kappa_h$
$$
\kappa_h = \max_{\mathbf{f}_h \in \mathbf{X}_h} \min_{\mathbf{p}_h\in \mathcal{RT}_h, \mathbf{v}_h \in \bV_{h}}  
\frac{\|\mathbf{p}_h - \nabla \mathbf{v}_h\|}{\|\mathbf{f}_h\|}\:.
$$
where the minimization respect to  $\mathbf{p}_h$ is subject to the condition 
$$
 {\nabla} \cdot \mathbf{p}_h + \nabla \phi_h +  \mathbf{f}_h=0 \text{ for certain } \phi_h \in U_h\:.
$$
Notice that for each $\bff_h \in \bX_h$, the minimization of $\|\mathbf{p}_h - \nabla \mathbf{v}_h\|$ is to find $\bu_h$ and $\bp_h$ that minize $\|\nabla \bu - \bu_h\|$ and $\|\nabla \bu - \bp_h\|$, respectively. The computation of $\kappa_h$ requires  to solve a matrix eigenvalue problem.

\medskip

In Thereom \ref{thm:a-priori-est-for-stokes}, we conostruct an explicit {\em a priori} error estimation for $\bu_h$.
\begin{theorem}
\label{thm:a-priori-est-for-stokes}
Given $\bff \in (L^2(\Omega))^3$, 
let $u$ be the solution of \eqref{eq:stokes}.
Then, 
\begin{equation}
\label{eq:c_h}
\| \B{u} -  P_h \bu  \| \le C_h \|\nabla (\B{u} -  P_h \bu ) \|,
\quad
\|\nabla (\B{u} - P_h \B{u}) \| \le C_h \|\B{f}\|\:,
\end{equation}
where $C_h :=\sqrt{\kappa_h^2 + C_{0,h}^2}$ and $C_{0,h}$ is the error estiamtion constant of  $\pi_h$ applied to $\bV$; see \eqref{eq:c_0_h_V}.

\end{theorem}
\begin{proof} Here is a sketch of the proof; see \cite{Liu-2021} for a full version.  
Let $\bff_h :=\pi_h \bu$ and $\overline{\bu}:= \Delta_s^{-1} \bff_h$. From the definition of $\kappa_h$ and the hypercircle (\ref{eq:hypercircle-eq}) in Lemma \ref{lemma:prager-synge} with $f$ replaced by $\bff_h$, we have
\begin{equation}
    \label{eq:local-est-kappah}
    \|\nabla (I-P_h)\overline{\bu} \| \le \kappa_h \|\bff_h\| \:.
\end{equation}
Also, by taking $\bv=(\bu-\overline{\bu})$ in the following equation,
$$
(\nabla (\bu-\overline{\bu}), \nabla \bv)
=(  \bff - \bff_h, \bv ) =  ( \bff - \bff_h, (I-\pi_h) \bv ) \le 
C_{0,h} \|\bff - \bff_h\| \cdot \|\nabla \bv \|\:,
$$
we have
\begin{equation}
\label{eq:local-est-ch}
    \|\nabla (\bu-\overline{\bu})\| \le C_{0,h} \|\bff - \bff_h\|
\end{equation}
Additionally, from the orthogonality of $P_h$, we have
 \begin{equation}
 \label{eq:local-est-best-approximation}
\|\nabla (\bu - P_h \bu)\| \le \|\nabla (\bu - P_h \overline{\bu})\| \le \|\nabla (\bu - \overline{\bu})\| +\|\nabla (\overline{\bu} - P_h \overline{\bu})\| \:. 
 \end{equation}
Thus, the estimation of $\|\nabla (\B{u} - P_h \B{u}) \|$ is available 
by combining \eqref{eq:local-est-kappah}, \eqref{eq:local-est-ch} and \eqref{eq:local-est-best-approximation}. The estimation of $\| \B{u} - P_h \B{u} \|$ can be obtained by further applying the Aubin-Nitsche technique.
\end{proof}

For any $\bff\in \bV$, we have
$$
\|\nabla(I-P_h) ( \Delta_s ^{-1} \mathbf{f} ) \|\le C_h \| \mathbf{f} \| \:.
$$
Since each $\mathbf{f}\in (L^2(\Omega))^3$ can be regarded as a functional in $\bV^\ast$ such that $\langle \bff, \bv \rangle:=(\bff, \bv)$, we have $\mathcal{A}^{-1} \bff = \Delta_s ^{-1} \bff/{\epsilon} \in \bV$ and 
\begin{equation}
\label{eq:a_priori_est_epsilon}
\|\nabla(I-P_h) (\mathcal{A}^{-1} \mathbf{f} ) \|\le C_h/\epsilon \| \mathbf{f} \| \:.
\end{equation}
Hence, we can take $C_{h, \mathcal{A}}=C_h/\epsilon$ in (\ref{eq:a_priori_est_epsilon_Ph}). 
By applying the $L^2$-norm error estimation of $P_h$ as given in (\ref{eq:a_priori_est_epsilon}), we have
$$
\|(I-P_h) (\mathcal{A}^{-1} \mathbf{f} ) \| \le
C_h \|\nabla(I-P_h) (\mathcal{A}^{-1} \mathbf{f} ) \| \le C_h^2 / \epsilon  \| \mathbf{f} \| \:.
$$

\subsection{Sub-problem b): Explicit eigenvalue estimation of differential operators}
\label{sec:sub-problem-b}

The eigenvalue problems appear in estimating various constants in the error analysis.
For example, the estimation of constant $C_{0,h}$ in the average interpolation error estimation is related to the eigenvalue problem of Stokes operator over tetrahedron elements with a Neumann boundary condition. 
Generally, it is difficult to give lower eigenvalue bounds for the operators over domains of general shapes. In \cite{Liu+Oishi2012,Liu2015, You-Xie-Liu-2019, Liu2020-JCAM}, the finite element methods are adopted to provide rigorous eigenvalue bounds in an efficient and easy-to-implement way. 
\medskip

The eigenvalue estimation method proposed in \cite{Liu2015} is stated under the following general function space settings. 
\begin{itemize}
    \item [(A1)] Let $V$ be a seperable Hilbert space $V$ and $V_h$ be its finite discretized space.  
    \item [(A2)] Let $a(\cdot, \cdot)$ and $b(\cdot, \cdot)$ be two strictly postivie symmetric bilinear forms over $V+V_h$. For the norms $\|\cdot\|_b$ and $\|\cdot\|_a$ introudced by the bilinear forms, assume that $\|\cdot\|_b$ is compact respect to $\|\cdot\|_a$.
\end{itemize}
Under the assumption (A1) and (A2), let us consider the following two eigenvalue problems.

(E) Objective eigenvalue problem: Find $u \in V$ and $\lambda \in R$ such that,
\begin{equation}
\label{eq:eig-variational-general}
a(u, v) = \lambda b(u, v),\quad \forall v \in V\:.
\end{equation}

($E_h$) Approximation to (E): Find $u_h \in V_h$ and $\lambda_h \in R$ such that,
\begin{equation}
    \label{eq:eig-variational-general-h}
    a(u_h, v_h) = \lambda_h b(u_h, v_h),\quad \forall v_h \in V_h\:.
\end{equation}
The eigenvalues of the eigenvalue problem (\ref{eq:eig-variational-general}) and (\ref{eq:eig-variational-general-h}) are denoted by  $\lambda_{1}\le \lambda_{2} \le \cdots $ and
 $\lambda_{h,1}\le \lambda_{h,2} \le \cdots \le \lambda_{h,n}$ ($n=\mbox{Dim}(V_h)$), 
 respectively.

The following theorem provides lower bounds for the objective eigenvalues of (E).
\begin{theorem}[Theorem 2.1 of \cite{Liu2015}] 
\label{thm:liu-eig-bound}
Suppose the following inequality holds for the projection $P_h:V \to V_h$ with respect to the inner product $a(\cdot, \cdot)$:
$$
\|u-P_h u\|_{b} \le C_{h,P_h} \|u-P_h u\|_{a},\quad \forall u \in V \:.
$$
Then, a lower bound for the objective eigenvalue $\lambda_i$ is given by
$$
\lambda_i \ge \frac{\lambda_{h,i}}{1+C_{h,P_h}^2\lambda_{h,i}},\quad (i=1,\cdots, n)\:.
$$
\end{theorem}

Below, let us apply Theorem \ref{thm:liu-eig-bound} to solve the  eigenvalue problem of Stokes differenital operator: Find $\bu :\Omega \to \mathbb{R}^3$,
$p: \Omega \to \mathbb{R}$ and $\lambda \in \mathbb{R}$ such that
\begin{equation}
\label{eq:eig-stokes-dirichlet}
-\Delta \bu +\nabla p= \lambda \bu, ~~ \Divv{\bu}=0 \mbox{ in } \Omega; \quad \bu=0 \mbox{ on } \partial \Omega.
\end{equation}
The variational formulation of (\ref{eq:eig-stokes-dirichlet}) is: Find $\bu \in \bV$ and $\lambda \in \mathbb{R}$ such that,
\begin{equation}
\label{eq:eig-stokes-variational}    
(\nabla \bu, \nabla \bv) = \lambda (\bu, \bv),\quad \forall \bv \in \bV\:.
\end{equation}
Let $\lambda_1$ be the first eigenvalue of the above eigenvalue problem, then $C_P(\Omega)=1/\sqrt{\lambda_1}$ is just 
the Poincar\'e constant that makes the following embedding equality holds.
$$
\|\bv\|_{L^2(\Omega)} \le C_P(\Omega) \|\nabla \bv\|_{L^2(\Omega)},\quad \forall \bv \in \bV\:.
$$ 

% \paragraph{Eigenvalue problem of Laplacian} Find $\lambda \in \mathbb{R}$ and $u$ such that 
% \begin{equation}
% \label{eq:laplacian-dirichlet}
% -\Delta u = \lambda u \mbox{ in } \Omega; \quad u=0 \mbox{ on } \partial \Omega.
% \end{equation}

In application of Thereom \ref{thm:liu-eig-bound}, we take the following function space setting.
$$
V:=\bV, \quad V_h:= \bV_h^{\CR}, 
$$
where $\bV_h^{\CR}$ is the  Crouzeix--Raviart FEM space with local divergence-free space property.  That is,
$$
\bV_h^{\CR} :=\{\bv_h \in \left(V_h^{\CR} \right)^3 \:|\: \Divv (\bv_h|_K) =0, \mbox{ for each }K\in \mathcal{T}^h \}\:.
$$
Here, $V_h^{\CR}$ is the Crouzeix-Raviart FEM space under the discretized homogeneous Dirichlet condition, i.e., every $v_h\in V_h^{\CR}$ 
has zero integral on each boundary facet of the mesh. 
Under this setting, \cite{Liu2015} provides an explicit value of $C_{h,P_h}$, which further leads to the following eigenvalue bounds.
$$
\lambda_{k} \ge \frac{\lambda_{h,k}}{1+(0.3804h)^2\lambda_{h,k} }\quad (k=1,2, \cdots, \mbox{Dim}(\bV_h^{\CR}))\:. 
$$
Here, $h$ is the largest diameter of the tetrahedra elements in the mesh.

For more results about bounding eigenvalues and various error constants, refer to \cite{Liu2017,liu-2013-2,Kikuchi+Liu2007_2,Kobayashi2011,Kobayashi2015}.

% Here is list of other eienvalue problems needed in evaluting the quantities in the solution verification.
% \begin{itemize}
%     \item $C_{0,h}$
%     \item $\lambda_1$: To be used in estimating $G$. See ... 
% \end{itemize}
% For a general disucssion of eigenvalue estimation, refer to \cite{Liu2020-JCAM}.

\subsection{Sub-problem c): Norm estimation of linear operator under divergence-free condition}
\label{sec:sub-problem-c}
% \mynote{Several things to be remark: 
% is tau must larger or equal to 1?}

The evalution of $\tau$ is to  
evalute the norm of operator $T^{-1}$ (see definition in (\ref{eq:def-T})) over the Scott-Vogelius FEM space $\bV_h$. 
However, since $\bV_h$ is not explicit constructed, it is difficult to evaluate 
$\tau$ from the definition by 
formulating the representation matrix for $T^{-1}$ under the basis of $\bV_h$.  
Here, let us introduce a lemma that helps to evaluate $\tau$ over space $\bU_{h,0}$, rather than divergence-free subspace subspace $\bV_h$.  

\medskip

The lemma below is formulated under a general function space setting.
\begin{lemma} []
Given Hilbert space $U$ with inner product $\langle \cdot, \cdot \rangle_U $, let $M:U\rightarrow U$ be a linear mapping $M$ and let $V:=M U \subset U$.
Suppose that the dual operator $M^\ast$ of $M$ satisfies $M^\ast U \subset V$. 
Then
$$
\max_{v \in V } \frac{\|Mv\|_{U}}{\|v\|_{U}}  = \max_{v \in U } \frac{\|Mv\|_{U}}{\|v\|_{U}} \:.
$$
\end{lemma}
\begin{proof}
Suppose that 
$$
t:= \max_{v \in U } R(v)=R(u), \quad (R(v):=\frac{\|Mv\|^2_U}{\|v\|^2_U})\:.
$$ 
Decompose $u$ by $u=u_1+u_2$ such that 
$u_1 \in V$ and $u_2 \in V^\perp$, where $V^\perp$ is the orthogonal complement subspace of $V$ in $U$. 
Then, 
\begin{eqnarray*}
\langle M u,  M u \rangle_U &=& \langle M u_1,  M u_1 \rangle_U + 2\langle M u_1,  M u_2 \rangle_U + \langle M u_2,  M u_2 \rangle_U \\
&=& \langle M u_1,  M u_1 \rangle_U + \langle M^\ast M (2 u_1+u_2),  u_2 \rangle_U \\
&=& \langle M u_1, M u_1 \rangle_U\:,
\end{eqnarray*}
since $M^\ast M ( 2u_1 + u_2) \in V$ and $u_2 \in V^\perp$.
Noticing that $\|u\|_U \ge \|u_1\|_U $, we have
$$
t= R(u) \le R(u_1) \le \max_{v \in V } R(v) \le \max_{v \in U }R(v)  = t \:.
$$
The above inequality leads to the conclusion.
\end{proof}

Let us define a mapping $T:V_h \to V_h$, 
\begin{equation}
    \label{eq:def-T}
T:={P}_{h}(I - {{\mathcal A}}^{-1}{\mathcal N}'[\hat{u}])|_{V_h}\:.
\end{equation}
Let $\bw_h:=T\bu_h$, then the following equation holds for $w_h$ and $u_h$.
$$
(\nabla \bw_h, \nabla \bv_h) = (\nabla \bu_h, \nabla \bv_h) - \frac{1}{\epsilon} (N'[\hat{\bu}]\bu_h, \bv_h),\quad \forall \bv_h \in \bV_h\:.
$$
The existence of the reverse of mapping $T$, i.e., $\bu_h = T^{-1} \bw_h$, can be confirmed by check the regularity of the matrices for the above equation.

Let $M:\bU_{h,0}\rightarrow \bV_h$ be the linear operator that for any $\bw_h \in \bU_{h,0}$, $\bu_h = M\bw_h \in \bV_h$ satisfies
\begin{equation}
    \label{eq:def-M-dual}
(\nabla \bu_h, \nabla \bv_h) - \frac{1}{\epsilon} (N'[\hat{\bu}]\bu_h, \bv_h) = (\nabla \bw_h, \nabla \bv_h) 
,\quad \forall \bv_h \in \bV_h\:.
\end{equation}
One can apply the Lagrange multiplier method to reformulate the problem: Find $\bu_h \in \bU_{h,0}$ and $p_h \in X_{h,0}$ such that
$$
(\nabla \bu_h, \nabla \bv_h) - \frac{1}{\epsilon} (N'[\hat{\bu}]\bu_h, \bv_h) +
(\Divv \bu_h , q_h) +(\Divv \bv_h , p_h) 
= (\nabla \bw_h, \nabla \bv_h) 
$$
for any $\bv_h \in \bU_{h,0}, q_h \in X_{h,0}$.

The dual operator of $M$ is given by: for any $\bw_h\in \bU_{h,0}$, 
$\bu_h^\ast :=M^\ast \bw_h\in \bV_h$ satisfies
\begin{equation}
\label{eq:def-M-dual}
(\nabla \bu_h^\ast, \nabla \bv_h) - \frac{1}{\epsilon} (\mathcal{N}'[\hat{\bu}]\bv_h, \bu_h^\ast) = (\nabla \bw_h, \nabla \bv_h) 
,\quad \forall \bv_h \in \bV_h\:.
\end{equation}
One can easily verify that, 
$$
T^{-1}=M|_{V_h},\quad 
{T^{-1}}^\ast=M^\ast|_{V_h}\:.
$$
Then the constant $\tau$ can be characterized by 
\begin{equation}
\tau := \max_{\bw_h \in \bV_h} \frac{ \|T^{-1}\bw_h\|_\bV }{\|\bw_h\|_\bV} 
= \max_{\bw_h \in \bU_{h,0}} \frac{ \|M \bw_h\|_\bV }{\|\bw_h\|_\bV}\:.
\end{equation}

\begin{remark}The formualtion of evaluation of $\tau$ over $\bU_{h,0}$ has the merit that the divergence-free condition is processed as constraint condition in the Lagrange multiplier method and the matries can be explicity constructed in solving the eigenvalue problem. The demerit of such kind of formulation is that the matrix in computation has a larger dimension as $\mbox{Dim}(\bU_{h,0}) + \mbox{Dim}(X_{h,0})$, ranther than its essential DOF as $\mbox{Dim}(\bV_h)=\mbox{Dim}(\bU_{h,0})-\mbox{Dim}(X_{h,0})$.
\end{remark}

Let $A$ be the matrix corresponding to the inner product of basis $\{\phi_i\}$ of $\mathbf{U}_{h,0}$, 
i.e., $A_{ij}=(\nabla \phi_i,\nabla \phi_j)$.
Let us use the same notation $M$ for the matrix corresponding to operator $M$ and $x$ the coefficient vector of $\bw_h \in \bU_{h,0}$, with respect to the basis of $\bU_{h,0}$. The the computation of ${\tau}$ is evaluated by solving the following matrix eigenvalue problem.
\begin{equation}
    \label{eq:tau-by-eigenpro-matrix}
M^TAMx = \lambda A x,\quad \tau = \max(\sqrt{\lambda}) \:.
\end{equation}

\section{Application of  Newton--Kantorovich's theorem}
 
This section shows the details in applying the Newton--Kantorovich theorem to solution verification for the Navier--Stokes equation. The content includes the construction of approximate solution $\hat{u}\in V_h$ and the estimation of the quantities $K$, $\delta$, and $G$.

\subsection{Step 1: Approximate solution  $\hat{\bu}\in \bV_h$ }

Using Newton's method, it is easy to find an approximation solution to the Navier--Stokes equation.
However, as a general numerical result, the approximation  $\bu_h$ may not satisfy the divergence-free condition strictly. Hence, a one-step correction is performed to get $\hat{\bu} \in \bV_h$ from the approximation $\bu_h \in  U_{0,h}$. 
To make sure the divergence-free condition strictly holds, the solution of the equation will be presented by an interval vector in the verified computing.
\medskip

Find $\hat{\bu} \in \mathbf{U}_{h,0}$ , $p_h \in X_h$ such that
$$
(\nabla \hat{\bu}, \nabla \bv_h) + (\Divv \hat{\bu}, q_h)  +  (\Divv \bv_h, p_h)=(\nabla \bu_h,\nabla \bv_h) \quad \forall \bv_h \in \mathbf{U}_{h,0}, q_h \in {X}_h
$$
The solution $\hat{\bu}$ of the above equation belongs to $\bV_h$. For a well approximate solution $\bu_h$, it is expected that $\hat{\bu} \approx \bu_h$.

\subsection{Step 2: Estimation of $K$}

\vskip 0.1cm

The estimation of $K$ is obtained by applying Nakao's method. Here, we show the details in evaluating the quantities $\nu_1,\nu_2,\nu_3$, $\tau$, and other involved constants. 
\paragraph{Poincar\'e constant $C_P(\Omega)$} Let $C_P(\Omega)$ be the Poincar\'e constant that satisfies
$$
\|\B{v}\|_{L^2(\Omega)} \le C_P(\Omega) \|\nabla \B{v}\|_{L^2(\Omega)} ~~~~ \forall \B{v}\in \B{V}(\Omega)\:.
$$
As stated in \S \ref{sec:sub-problem-b}, such a constant can be estimated by solving the eigenvalue problem of the Stokes operator in each tetrahedral element.
An upper bound of $C_P(\Omega)$ can also be selected as the one that makes the following inequality hold. 
$$
\|v\|_{L^2(\Omega)} \le \widehat{C}_P \|\nabla v\|_{L^2(\Omega)} ~~~~ \forall v \in H_0^1(\Omega)\:.
$$

\paragraph{\em Estimation of $\nu_1$}

Let $\bw_h:={P}_{h} {{\mathcal A}}^{-1} {\mathcal N}'[\hat{\bu}] u_c$,  Then 
$$
(\epsilon \nabla \bw_h, \nabla \bv_h) = \langle \mathcal N'[\hat{\bu}]\bu_c, \bv_h \rangle \quad \forall \bv_h \in \bV_h\:.
$$
Take $\bv_h:=\bw_h$ and apply the inequality 
$\| \mathcal N'[\hat{\bu}] \mathbf{u}_c \|_{V^\ast} \le \nu_3 \| \mathbf{u}_c\|_{\mathbf{V}}$,
$$
\epsilon \|\nabla \bw_h\|^2 \le \nu_3 \| \mathbf{u}_c \|_{V} \|\nabla \bw_h\| \:.
$$
Hence, we have an upper bound of $\nu_1$ as
\begin{equation}
\nu_1 \le \frac{\nu_3}{\epsilon}  \:.
\end{equation}

\paragraph{\em Estimation of $\nu_2$}
By applying the Schwartz inequality, we have,
$$
 ((\B{u} \cdot \nabla ) \hat{\B{u}}, \B{v}) \le 3 \| \nabla \hat{\B{u}}\|_{\infty}  \cdot \|\B{u}\| \cdot \|\B{v}\|
\le 3 C_P^2(\Omega) \| \nabla \hat{\B{u}}\|_{\infty}   \|\nabla \B{u}\| \cdot \|\nabla \B{v}\|\:,
$$
$$
((\hat{\bu} \cdot \nabla ) {\B{u}}, \B{v}) 
\le \sqrt{3}\| \hat{\bu} \|_{\infty} \cdot \|\nabla \bu \|\cdot \|\bv\| 
\le \sqrt{3} C_P(\Omega) \| \hat{\bu} \|_{\infty}\cdot \|\nabla \bu \|\cdot \|\nabla \bv\| \:.
$$
Thus,
$$
|\langle {\mathcal N}'[\hat{\B{u}}] \B{u}, \B{v}\rangle |
%= ((\B{u} \cdot \nabla ) \hat{\B{u}}, \B{v})  +  ((\hat{\bu} \cdot \nabla ) {\B{u}}, \B{v}) 
\le 
(3 C_P^2(\Omega) \| \nabla \hat{\B{u}}\|_{\infty} + \sqrt{3} C_P(\Omega) \| \hat{\bu} \|_\infty)
\|\bu\|_V\|\bv\|_V\:.
$$
Hence, we have an upper bound of $\nu_2$ as
\begin{equation}
    \nu_2 \le 3 C_P^2(\Omega) \| \nabla \hat{\B{u}}\|_{\infty} + \sqrt{3} C_P(\Omega) \| \hat{\bu} \|_\infty \:.
\end{equation}

\paragraph{\em Estimation of $\nu_3$}
For $\B{u}_c=u-P_h u$, noticing that 
$ \|\B{u}_c\| \le C_h  \|\nabla \B{u}_c\|$, we have
\begin{eqnarray*}
|( {\mathcal N}'[\hat{\B{u}}] \B{u}_c, \B{v})| 
&=& 
| ((\B{u}_c \cdot \nabla ) \hat{\B{u}}, \B{v})  +  ((\hat{\bu} \cdot \nabla ) {\B{u}}_c, \B{v})| \notag\\
&=&
|((\B{u}_c \cdot \nabla ) \hat{\B{u}}, \B{v})  - ((\hat{\bu} \cdot \nabla )\B{v},  {\B{u}}_c) |\\
&\le&
(3  C_P \| \nabla \hat{\B{u}}\|_{\infty}  
+ \sqrt{3}\| \hat{\bu} \|_\infty)
\cdot \|\B{u}_c\|  \cdot \|\nabla \B{v}\| \\
&\le &
(3  C_P \| \nabla \hat{\B{u}}\|_{\infty}  
+ \sqrt{3}\| \hat{\bu} \|_\infty)
\cdot C_h  \|\nabla \B{u}_c\|  \cdot \|\nabla \B{v}\|\:.
\end{eqnarray*}
Therefore, an upper bound of $\nu_3$ is given by
\begin{equation}
\label{eq:nu_3_estimation}    
\nu_3 \le (3  C_P \| \nabla \hat{\B{u}}\|_{\infty}  
+ \sqrt{3}\| \hat{\bu} \|_\infty)
\cdot C_h \:.
\end{equation}

\vskip 0.5cm

\paragraph{\em Estimation of $\tau$}

The evaluation of $\tau$ is one of the core part of the solution verification. Since a small value of $\epsilon$ will cause a large value $C_{h,\mathcal{A}}$ and $\tau$, thus 
the condition that $\kappa=(\tau \nu_1 \nu_2 + \nu_3)C_{h,\mathcal{A}}<1$ may not hold and the evaluation of $K$ will fail. 
Notice that $\kappa = O(h^{2r})$ if $C_h=O(h^r)$, where the convergence rate $0<r<1$ is determined by the solution regularity of the Stokes equation over given domain; for a convex domain, we have $r=1$.
Therefore, to have a successful verification even for small $\epsilon$, one can refine the mesh to have a smaller value of $C_{h,\mathcal{A}}$ such that $\kappa <1 $. The details of evaluation of $\tau$ is explained in \S \ref{sec:sub-problem-c}.

%~ Let the basis function of $V_h$ be $\{\phi_i\}_{i=1,n}$ and $A_0:=\left( (\nabla \phi_i, \nabla \phi_j)\right)_{n\times n}$ 

%~ Suppose $w_h=\sum_{i}\phi_iy_i$,$u_h=\sum_{i}\phi_ix_i$. The matrix corresponding to mapping $T$ is still denoted by $T$. Then, $y=Tx$ and $\tau$ can be evaluatd by solving eigenalue problem of matrix.
%~ $$
%~ \tau^2 = \max_{y} \frac{x^TA_0x}{y^TA_0y} = \max_{y} \frac{y^T K^{-T} A_0 A^{-1}y}{y^TA_0y}
%~ $$

\subsection{Step 3: Estimation of $\delta$}
Let us introduce $\mathbf{g}:=(\hat{\bu}\cdot \nabla ) \hat{\bu} - \mathbf{f}$.  To evalute the residue error of $\mathcal{F}[\hat{\bu}]$, let us seek  $\mathbf{p}_h \in \mathbf{RT}_h$ such that $\mathbf{p}_h \approx \nabla {\hat{\bu}}$ by taking $\mathbf{p}_h$ as the minimizer of 
\begin{equation}
\label{eq:min-problem-for-ph}
\min_{\mathbf{p_h\in \mathbf{RT}_h}} \|\mathbf{p}_h - \nabla \hat{\bu}\|\:,
\end{equation}
where $\mathbf{p}_h$ is subject to the constraint condition:
\begin{equation}
\label{eq:ph_condition_in_delta_estimation}
(\epsilon ~\mbox{div } \mathbf{p}_h - \mathbf{g} , \mathbf{q_h}) =0, ~~\forall \mathbf{q}_h \in X_h \:.
\end{equation}
%
%
%~ The variational problem to obtain $\mathbf{p}_h$ is : 
 %~ Find $\mathbf{p}_h \in \mathbf{RT}_h$, $\mathbf{\rho}_h \in \mathbf{X}_h$, $\phi_h\in \mathbf{U}_h$, $c\in \mathbb{R}$ , s.t.,
%~ \begin{align}
%~ (\mathbf{p}_h, \mathbf{q}_h) + \epsilon ~(\mathbf{\rho}_h, \nabla\cdot \mathbf{q}_h) + \epsilon ~(\nabla \cdot \mathbf{p}_h, \mathbf{\eta}_h) + (\nabla \phi_h, \mathbf{\eta}_h)+ (\mathbf{\rho}_h, \nabla \psi_h) + (\phi_h, d) + (\psi_h,c)\\
%~ = (\nabla \hat{\mathbf{u}}, \mathbf{q}_h) + ((\hat{u}\cdot \nabla ) \hat{u} - \mathbf{f},\eta_h)
%~ ~~~~~~~\forall \mathbf{q}_h \in \mathbf{RT}_h, \eta_h \in \mathbf{X}_h, \psi_h\in \mathbf{U}_h, d \in \mathbb{R} ~~~~~
%~ \end{align}

Let $\widehat{\mathbf{p}}_h$  be an approximate solution to the  minimizaiton problem (\ref{eq:min-problem-for-ph}). 
\begin{eqnarray}
\langle \mathcal{F}[\hat{\bu}] , v\rangle  
\!&\!\!=\!\!&\epsilon ~(\nabla \hat{\bu} - \widehat{\mathbf{p}}_h, \nabla \bv) + \left(  \mathbf{g} - \epsilon  \mbox{ div } \widehat{\mathbf{p}}_h , \bv \right) \notag \\
&\!\!=\!\!&\epsilon ~(\nabla \hat{\bu} - \widehat{\mathbf{p}}_h, \nabla \bv) + ( (I-\pi_h) \mathbf{g}  + (\pi_h  \mathbf{g} - \epsilon ~ \Divv \widehat{\mathbf{p}}_h), \bv) \notag 
\label{eq:residue_error_estimate}
\end{eqnarray}
By applying the error estimation of $\pi_h$ and the Poincar\'{e} constant, we have 
\begin{equation}
\label{eq:delta_estimation}    
\delta = \|\mathcal{F}[\hat{\bu}]\|_{V^\ast} \le \epsilon \|\nabla \hat{\bu}- \widehat{\mathbf{p}}_h \| + C_{0,h} \| (I-\pi_{h})  \mathbf{g}
\|
+C_{p} \|\pi_h \mathbf{g}  - \epsilon ~ \Divv \widehat{\mathbf{p}}_h \|.
\end{equation}
Notice that constraint condition (\ref{eq:ph_condition_in_delta_estimation}) implies $\pi_h \mathbf{g}=\epsilon \:\Divv \mathbf{p}_h$. 
Since the minimization problem (\ref{eq:min-problem-for-ph}) can be easily solved by classical numerical schemes, the last term in (\ref{eq:delta_estimation}) will be pretty small compared to other terms.

\subsection{Step 4: Estimation of $G$}
\label{subsec:est-G}
Notice that for general $\bv, \bu,\tilde{\bu} \in \bV$,
$$
|(( \bv \cdot \nabla ) \bu, \tilde{\bu})| \le \sqrt{3} \| \bv\|_{L^4} \|\nabla \bu\| \|\tilde{\bu}\|_{0,4} \le \sqrt{3}C_{4,P}^2\|\nabla \bv\| \cdot \|\nabla {\bu}\| \cdot  \|\nabla \tilde{\bu}\| \:,
$$
where $C_{4,P}$ is defined by
$$
C_{4,P} := \max_{\bv\in \bV } \frac{ \|\bv\|_{L^4}}{\|\nabla \bv\|}\:.
$$
By following Plum's result in Lemma 2 of \cite{Plum2008}, we have an upper bound of $C_{4,P}$ as $$
C_{4,P} \le \left(\frac{8}{9}\right)^{1/4}\cdot \left(\frac{1}{\lambda_1}\right)^{1/8}\:,
$$
where $\lambda_1$ is the first eigenvalue of the Stokes operator with the  homogeneous Dirichlet boundary condition.
Notice that the analysis in \cite{Plum2008} only considers the case of $H^1_0(\Omega)$, while its results can be easily extended to the divergence-free space $\bV$.
Thus, 
$$
|\langle (\mathcal{F}'[\bv]  - \mathcal{F}'[\bw]) \bu,\tilde{\bu} \rangle |= |\langle \mathcal{N}'[\bv-\bw]\bu,\tilde{\bu} \rangle |
\le 2 \sqrt{3} C_{4,P}^2 \|\nabla (\bv-\bw)\| \cdot \|\nabla {\bu}\| \cdot  \|\nabla \tilde{\bu}\|\:,
$$
which implies $G$ can be taken as
\begin{equation}
    \label{eq:est-G}
G =\frac{4\sqrt{6}}{3}\left(\frac{1}{{\lambda_1}}\right)^{1/4}\:.
\end{equation}

\section{Solution verification example}  

In this section, a solution verification case over a 3D non-convex domain is reported. The computing is performed under the HPC server provided by Oishi Lab at Waseda University, and also the Ganjin online computing environment \cite{ganjin}.
The solution verification example reported here can be directly tested online at the Ganjin site.

\subsection{Mesh generation and interpolation error constants}
\label{subsec:mesh-generation}

Since the Scott-Vogelius FEM for Stokes equation over a general mesh will lead to the singularity of discretized matrices, we apply Zhang's method \cite{Zhang} in the mesh generation process to avoid the singularity.
Note that for 2D case, a mesh without a degenerate point is required for a stable computation \cite{scott1985norm}. 

 In our computation experiments, the domain will be firstly subdivided into small uniform cubes. Then each cube is divided into $5$ tetrahedra to obtain a standard mesh $\mathcal{T}^h$. To construct the Scott-Vogelius space, the mesh 
 $\mathcal{T}^h_{SV}$ is refined by following Zhang's method to have $\mathcal{T}^h_{SV}$, that is, each tetrahedron of $\mathcal{T}^h$ is further partitioned into $4$ sub-tetrahedra with respect to the barycentric of the tetrahedron. The spaces $\bRT_h$ over $\mathcal{T}^h$ and $\bV_h$ over $\mathcal{T}^h_{SV}$ have independent meshes in the construction of the hypercircle for the {\em a priori } error estimation. 

 Let $h$ be the edge length of small cubes in the subdivision. From the results in \cite{JCAM-LIU-2020}, we have an upper bound for the Poincar\'e constant over the tetrahedra of $\mathcal{T}^h$ resulted from our mesh generation method. 
\begin{equation}
  \label{eq:poincare-constant-bound}
  C_{0,h} \le \widehat{C}_{0,h} \le 0.284 h ~~ (h: \mbox{the largest edge length of sub-cubes}) .
\end{equation}

\medskip

In the following computation example, the degrees of FEM function spaces are selected such that $k-1=d=m=2$. 
\subsection{VFEM package for rigorous computation}
Implementing the solution verification algorithm proposed in this paper is not easy work. The existing FEM libraries, such as FEniCS, can give efficient computation for various partial differential equations. 
However, such libraries are not suitable for verified computing since the process of matrix assembling introduces errors such as the rounding error of floating-point numbers and the approximation error of numerical quadrature.  
The first author developed the Verified Finite Element Method (VFEM) MATLAB toolbox to assemble the matrices with rigorous computation. 
To bound the rounding error, the interval arithmetic is utilized. 
To have rigorous integral of polynomial basis of FEM spaces, the base functions are represented by Bernstein polynomials with the volume coordinate $(u,v,w,t)$ over each tetrahedra element $K$. 
$$
B^N_{ijkl}(u,v,w,t) := \frac{N!}{i!j!k!l!}  u^iv^jw^kt^l, \quad % \quad \binom{N}{i,j,k,l}:=, 
(i+j+k+l=N, u+v+w+t=1)\:.
$$
The following formula is helpful in calculating the integral of polynomials. 
$$
\int_K  u^iv^jw^kt^l \mbox{d} K  = 6|K| \frac{i!j!k!l!}{(i+j+k+l+3)!} \mbox{ for } i,j,k,l \ge 0\:.
$$
Different from the standard method, which considers the transformation between the objective element $K$ and a reference $\hat{K}$, VFEM calculates the explicit value of integrals of polynomial systems on each element directly.

The VFEM package has an implementation of the continuous Lagrange FEM space, the Raviart--Thomas FEM space, the Courzeix-Riviart FEM space. It provides the general function operators such as the gradient operator and the divergence operator. 
The basic operations for Berstein polynomial, such as the domain subdivision under the de Casteljau scheme, the degree raising operation, are also available. The selection of Bernstein polynomial also enjoys the efficiency in evaluating the $L^\infty$ norm of FEM solutions by using its convex hull property. The package is provided as MATLAB codes and has an interface to switch between classical approximate calculation and the verified computing with the INTLAB toolbox \cite{INTLAB}. In the near future, the package with C++ language support will also be developed.

The matrices in the FEM computation are assembled by VFEM. To have rigorous computation results of linear systems or matrix eigenvalue problems, one has to turn to the functions provided by the INTLAB toolbox or self-developed algorithms.

\subsection{Solution verification example over a 3D non-convex domain}
Let us illustrate a solution verification example over a 3D domain. The domain $\Omega$ is selected as a cuboid with a hole inside:
$$
\Omega = ((0,1)^2 \setminus [0.25,0.5]^2)\times (0,{0.5})\:.
$$
The detailed setting for the numerical example is as follows:
$$
\mathbf{f}=(15(1-y)^2,0, 10z^2), ~~ \epsilon = 0.25\:.
$$
A rigorous estimatin of the norm of $\bff$ tells that $\|\bff\|_{L^2} \in [4.6826, 4.6827]$. 

\medskip

The mesh and FEM spaces are created in the following way:
\begin{itemize}
    \item Along x-, y-direction, the longest edge of the domain boundary is uniformly divided to $N=4$ parts, and along the z-direction, the edge is divided into $N/2=2$ parts. Thus total $30\times3$ small cubes with the size of $0.25$ are obtained. Mesh $\mathcal{T}^h$ is obtained by dividing each block into $5$ tetrahedra along the diagonal lines. Finally, by following Zhang's method \cite{Zhang}, $\mathcal{T}^h_{SV}$ is created by further dividing each tetrahedron of $\mathcal{T}^h$ into $4$ sub-tetrahedra with respect to its centroid. The mesh $\mathcal{T}^h_{SV}$ contains total $600$ tetrahedron elements.
    \item The degree of FEM function spaces is selected as: $d=2, m=2, k=3$. That is, the discontinuous space $X_h$ and 
    the Raviart--Tomas space $RT_h$ over $\mathcal{T}^h$ are constructed by piecewise polynomial with degree as $d=2$ and $m+1=3$, respectively. 
    The Scott--Vogelius space $\bV_h$ over $\mathcal{T}^h_{SV}$ 
    has polynomial degree as $k=3$. 
    Denote the dimension of the matrix for constructing $\bV_h$ by $N_{Conf}$ and the dimension for the matrix in mixed formulation by $N_{Mix}$.
    The dimension of spaces are list in Table \ref{tab:fem-sapce-dim}. Note that independent mesh selection for Scott--Vogelius space and Raviart--Thomas space in constructing the hypercircle equation makes it possible for a balance of matrix size in the computation in the sense that 
    $N_{Conf}\approx N_{Mix}$.
    
    \begin{table}[h]
        \centering
        \begin{tabular}{|c|c|c|c|c|c|c|c|c|}
        \hline
        \rule[-3mm]{0cm}{8mm}{Mesh} &  \multicolumn{4}{|c|}{$\mathcal{T}^h_{SC}$} & \multicolumn{4}{|c|}{$\mathcal{T}^h$}\\
        \hline
        \rule[-3mm]{0cm}{8mm} Space & $\bU_{h,0}$ & $X_{h,0}$ & $\bV_h$ & $N_{Conf}$ & $\bRT_h$ & $\bX_{h}$ & $U_{h}$ & $N_{Mix}$ \\
        \hline
        \rule[-3mm]{0cm}{8mm}Dim & 7965 & 5999 & 1966 & 13964 & 12060 & 4500 & 1037 & 17596 \\  
        \hline
        \end{tabular}
        \caption{Dimension of FEM spaces and matrices}
        \label{tab:fem-sapce-dim}
    \end{table}
\end{itemize}

The approximate solution $\hat{\bu}$ to Navier--Stokes equation is obtained by using the solver from FEniCS and then loaded into VFEM as a piecewise polynomial. To have the divergence-free condition strictly satisfied, one projection of $\hat{\bu}$ into divergence-free Scott--Vogelius FEM space is performed.
 
The streamlines of the approximate solution $\hat{\bu}$ for the velocity field is displayed in Fig. \ref{fig:computation_result}. Below is the property of the approximation solution.
$$
\|\hat{\bu}\|_{L^2} = 0.0356,~~\|\nabla \hat{\bu}\|_{L^2} = 0.4543,~~
\|\hat{\bu}\|_{L^\infty} = 0.1466,~~
\|\nabla \hat{\bu}\|_{L^\infty} = 11.8763 \:.
$$
The mesh for the computation is very raw, but its size is well selected so that the corresponding numerical computation can be easily solved in an entry-level workstation with 64 GB memory. 
A further refined mesh will require dramatically increased resources in solution verification.

\begin{figure}[ht]
\begin{center}
\includegraphics[scale=0.65]{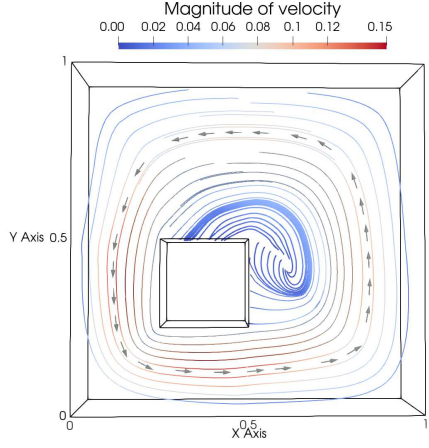}
\caption{\label{fig:computation_result} Streamline of the approximate solution around the z=0.25 plane}
\end{center}
\end{figure}

\paragraph{Values of various quantities used in the solution verification}

\begin{itemize}

\item The minimal eigenvalue $\lambda_1$ of the Stokes operator with homogeneous boundary condition and 
the Poincar\'e constant $C_{p}$ are obtained by applying the Crouzeix--Raviart FEM space to the  twice refined mesh of $\mathcal{T}^h$: 
% $$
% \lambda_1>114,~~ C_P=1/\sqrt{\lambda_1} = 0.0932\:.
% $$
$$
\lambda_1>139.60,~~ C_P=1/\sqrt{\lambda_1} \le 0.0846\:.
$$
An upper eigenvalue bound can be obtained as $\lambda_1 < 159.1 $ by utilizing the Scott-Vogelius FEM over $\mathcal{T}^h_{SV}$.
Note that by considering the eigenvalue of Laplacian over the cuboid without a hole, one can also obtain 
a analytical and very raw eigenvalue bound as
 $\lambda_1 \ge 6 \pi^2 \approx 59.2$ 
 with the Laplacian's eigenfunction
 $u=\sin\pi x \sin\pi y\sin 2\pi z $.
 
\item Error constants in {\em a priori} error estimation: 
    $$
    C_{h} =\sqrt{\kappa_h^2+C_{0,h}^2} =\sqrt{0.0583^2 + 0.0625^2}\le 0.08548;~
    C_{h,\mathcal{A}} =\frac{C_h}{\epsilon} \le 0.3568 \:.
    $$

\item Estimate of $K$ (norm of inverse operator $\mathcal{F}'[\hat{\bu}]^{-1}$): 
$$
\nu_1=1.1663, ~\nu_2=0.2768, \nu_3=0.2916, ~\tau =1.0016,~\kappa = 0.2793, ~K=2.2090 \:. 
$$

\item Estimation for the residue error of $\mathcal{F}[\hat{\bu}]$ and the local continuity:
    $$\delta=0.05743,~~G\le 0.9502 \:.$$

\item 
The solution existence condition by Newton--Kantorovich's theorem:
  $$
  \alpha \cdot \omega =   K\delta \cdot  KG = 0.2663 < 1/2 \:.
  $$
\end{itemize}

\paragraph{Conclusion}

From the Newton--Kantorovich theorem, we can declare the stationary solution existence and uniqueness of the Navier--Stokes equation inside the ball $B(\hat{\bu},\rho)$, where
$$
\rho = \frac{ 1- \sqrt{1-2\alpha \omega} }{\omega}= 0.1507\:.
$$

\begin{remark} Let us apply Girault-Raviart's theorem in section 1 to the problem considered here.
Since ${N}$ is difficult to evaluate, we apply the technique for estimating $G$ to have a theoretical upper bound as: 
$$
{N} \le G/2 \approx 0.4751 \:.
$$
The equality $(\bff,\bv)\le C_P \|\bff\| \cdot \|\bv\|_\bV$ leads to an estimation of $\|\bff\|_{V^\ast}$:
$$\|\bff\|_{V^\ast}\le C_P \|\bff\|\le 0.3961\:.
$$
One can apply the Scott--Vogelius FEM to find an approximation to $\Delta_s^{-1}\bff$ and then apply the 
{\em a priori} error estimation to estimate $\|\bff\|_{V^\ast}$ directly.  
By further refining current mesh $\mathcal{T}^h_{SC}$ for 3 times and applying the Scott--Vogelius FEM, we have 
$$
\|\Delta_s^{-1}\bff\|_{\bV} \le \|P_h(\Delta_s^{-1}\bff) \|_{\bV} +
\|(I-P_h)(\Delta_s^{-1}\bff))\|_\bV \le 0.131 + 0.051 \le 0.182\:.
$$
With $\epsilon=0.25$, we have 
$$
{N} \cdot \|f\|_{V^\ast}/\epsilon^2 \le
0.4751\cdot 0.182/0.25^2\approx 1.383 ~~( >1 )\:.
$$
Therefore, the unique solution existence cannot be esaily confirmed by only using Girault-Raviart's theorem.
\end{remark}

\section*{Conclusion}

In this paper, we propose a method to provide solution verification for the stationary solution of the Navier--Stokes equation. The method is based on the finite element approximation for objective function spaces, and the algorithm can easily deal with 3D domains of general shapes. Since the simulation in 3D domains usually causes large-scaled matrices and the rigorous computation requires higher computing resources compared to classical approximate computation, for the moment, only the flow with small Reynolds numbers can be verified in a reasonable time. In the following research, we will continue improving the efficiency for both the theoretical error analysis and the code implementation in rigorous computations. It is aimed to verify the solution to Navier--Stokes's equation with large Reynolds numbers to investigate complex phenomena of 3D flows.

\vskip 0.5cm

\bibliographystyle{plain}

  %%%%% References
  \bibliography{mylib}

  %%%%% References
  %~ \begin{thebibliography}{9}
  
  %~ \bibitem{Girault_Raviart_1986} V. Girault and  P.A. Raviart, Finite Element Methods for Navier--Stokes Equations: Theory and Algorithms, 1986, Springer-Verlag.
  %~ \bibitem{SIAM-LIU}
  %~ Xuefeng Liu and Shin'ichi Oishi, Verified eigenvalue evaluation for Laplacian over polygonal domain of arbitrary shape, SIAM J. Numer. Anal., 51(3), pp.1634-1654, 2013.
  %~ \bibitem{Zhang}
  %~ Shangyou Zhang, A new family of stable mixed finite elements for the 3D Stokes equations, Mathematics of computation 74(250) pp.543-554, 2005.
  %~ \end{thebibliography}
  
%  %%%% Appendix
%  \appendix
%  \section{Appendix}
\end{document}